\documentclass[12pt]{article}
\usepackage{amsfonts}
\usepackage{amsmath}
\usepackage{amssymb}
\usepackage[mathscr]{eucal}
\usepackage{graphicx}
\usepackage{fullpage}
\usepackage{times}
\usepackage{epsfig}
\usepackage{color}
\usepackage[pagebackref]{hyperref}    
\usepackage[dvipsnames]{xcolor}
\definecolor{ao}{rgb}{0.0, 0.5, 0.0}
\definecolor{darkmagenta}{rgb}{0.55, 0.0, 0.55}

\def\eod{\vrule height 6pt width 5pt depth 0pt}
\newenvironment{proof}{\noindent {\bf Proof:} \hspace{.2em}}
{\hspace*{\fill}{\eod}}

\newtheorem{theorem}{Theorem}
\newtheorem{lemma}[theorem]{Lemma}
\newtheorem{definition}[theorem]{Definition}
\newtheorem{corollary}[theorem]{Corollary}

\newtheorem{example}[theorem]{Example}

\newcommand{\sB}{\mathcal{B}}

\newcommand{\od}{ \overline{ d}}

\newcommand{\sL}{\mathcal{ L}}

\newcommand{\RR}{ \mathbb{R}}
\newcommand{\ZZ}{ \mathbb{Z}}
\newcommand{\QQ}{ \mathbb{Q}}

\newcommand{\vsp}{\vskip 1em}

\newcommand{\ch}{ \mathsf{ch}}

\newcommand{\CF}{ \mathsf{CF}}

\newcommand{\rhalf}{\lfloor r/2 \rfloor}
\newcommand{\nhalf}{\lfloor n/2 \rfloor}

\newcommand{\comment}[1]{}

\newcommand{\SSS}{\mathfrak{S}}

\newcommand{\CC}{ \mathbb{C}}

\newcommand{\GTS}{\mathsf{GTS}}
\newcommand{\law}{\mathsf{Lexaway}}
\newcommand{\awy}{\mathsf{Aw}}
\newcommand{\BT}{\mathsf{BT}}
\newcommand{\wt}{\mathsf{wt}}

\newcommand{\id}{\mathsf{id}}

\title{Generalized Matrix polynomials of Tree Laplacians indexed 
by Symmetric functions and the $\GTS$ poset}

\begin{document}
\author{Mukesh Kumar Nagar\\
	Department of Mathematics\\
	Indian Institute of Technology Kanpur \\
	Kanpur 208 016, India.\\
	email: mukesh.kr.nagar@gmail.com
	\and
	Sivaramakrishnan Sivasubramanian\\
	Department of Mathematics\\
	Indian Institute of Technology Bombay\\
	Mumbai 400 076, India.\\
	email: krishnan@math.iitb.ac.in
}

\maketitle

\begin{abstract}
Let $T$ be a tree on $n$ vertices with $q$-Laplacian $\sL_T^q$ and Laplacian matrix $L_T$. 
Let $\GTS_n$ be the generalized tree shift poset on the set of unlabelled 
trees on $n$ vertices.  Inequalities are known between coefficients of the
immanantal polynomial of $L_T$ (and $\sL_T^q$) as we go up the poset $\GTS_n$.
Using the Frobenius characteristic, this can be thought as a result involving the 
schur symmetric function $s_{\lambda}$.  In this paper, we use an arbitrary
symmetric function to define a {\it generalized matrix function} of an $n \times n$ matrix.
When the symmetric function is the monomial and the forgotten symmetric function,
we generalize such inequalities among coefficients of the generalized matrix polynomial of 
$\sL_T^q$ as we go up  the $\GTS_n$ poset.
\end{abstract}

{\it Keywords:} Tree, $\GTS_n$ poset, Laplacian, 
monomial symmetric function, generalized matrix polynomial.  

\vsp

{\it AMS Subject Classification:} 05C05, 06A06, 15A15

\section{Introduction}
\label{sec:intro}

For a positive integer $n$, let $[n]=\{1,2,\ldots,n\}$ and  let
$\SSS_n$ denote the symmetric group on the set $[n]$.
We denote partitions $\lambda$ of the number $n$ as $\lambda \vdash n$. 
We write partitions using the exponential notation, with multiplicities of parts written as exponents.
For $\lambda \vdash n$, let $\chi_{\lambda}^{}$ be the irreducible 
character of $\SSS_n$ 
over $\CC$ indexed by $\lambda$. We think of $\chi_{\lambda}^{}$ 
as a function $\chi_{\lambda}^{} : \SSS_n \mapsto \ZZ$.
With respect to an irreducible character 
$\chi_{\lambda}^{}$, 
define the immanant of the $n \times n$
matrix $A = (a_{i,j})_{1 \leq i,j \leq n}$ as 
\begin{eqnarray}
  \label{eqn:immanant}
  d_{\lambda}(A) & = & \sum_{\psi \in \SSS_n} \chi_{\lambda}^{}(\psi) 
  	\prod_{i=1}^n a_{i,\psi(i)}. 
\end{eqnarray}

Let 
$\Lambda_{\QQ}^n$ denote the vector space of degree $n$ symmetric functions with 
coefficients from $\QQ$.  
The set of monomial symmetric functions $\{m_{\lambda} \}_{\lambda\vdash n}$ is one of the well known bases of $\Lambda_{\QQ}^n$. 
We refer the reader to the books by
Stanley \cite{EC2} and  by Mendes and Remmel \cite{mendes-remmel-book} for background on
symmetric functions.  $\Lambda_{\QQ}^n$ actually has an inner product structure as well
(see  Stanley \cite{EC2}). Another inner product space often studied is $\CF_n$,
the space of class functions from $\SSS_n \mapsto \CC$.  Further, there is a well
known isometry between these two spaces called the Frobenius characteristic, denoted  
$\ch: \CF_n \rightarrow \Lambda_{\QQ}^n$
(see Stanley \cite{EC2}).

Let $\gamma \in \Lambda_{\QQ}^n$ and consider $\Gamma_{\gamma} = \ch^{-1}(\gamma)$
its inverse Frobenius image.  Clearly, $\Gamma_{\gamma} \in \CF_n$ is a class function 
indexed by $\gamma$.
Define the generalized matrix function (GMF henceforth) of an $n \times n$ matrix $A$ with respect
to $\gamma$ as 
\begin{equation}
  \label{eqn:gen_matrix_fn}
  d_{\gamma}(A) = \sum_{\psi \in \SSS_n} \Gamma_{\gamma}(\psi) \prod_{i=1}^n a_{i,\psi(i)}.
\end{equation} 

Define the generalized matrix polynomial $\zeta_{\gamma}^A(x)$ of 
$A$ in a new variable $x$ with respect to a symmetric function $\gamma$ as 
follows:  $\zeta_{\gamma}^A(x)=d_{\gamma}(xI-A)$.  
It is well known  that the inverse Frobenius image $\ch^{-1}(s_{\lambda})$ 
of the Schur symmetric function $s_{\lambda}$ is $\chi_{\lambda}^{}$,
the irreducible character of $\SSS_n$ over $\CC$ indexed by $\lambda$  (see \cite{EC2}). 
Thus, from \eqref{eqn:immanant} and \eqref{eqn:gen_matrix_fn}, we can see that  $d_{s_{\lambda}}(A)=d_{\lambda}(A)$ and $\zeta_{s_{\lambda}}^A(x)=d_{\lambda}(xI-A)$.
 
Csikv{\'a}ri in \cite{csikvari-poset1} defined a poset 
on the set of unlabelled trees with $n$ vertices that we
denote in this paper as $\GTS_n$.  Among other results, he showed that 
if one goes up along the $\GTS_n$ poset, all coefficients of the 
characteristic polynomial of the Laplacian matrix $L_T$ of a tree 
$T$ decrease in absolute value.
This result was generalized by Mukesh and Sivasubramanian in \cite{mukesh-siva-gts} to 
immanantal polynomials of $\sL_T^q$ indexed by any $\lambda \vdash n$, where $\sL_T^q$ is
the $q$-Laplacian matrix of $T$ (see Theorem \ref{thm:main_earlier}).

Let $\RR^+$ denote the set of non-negative real numbers and $\RR^+[q^2]$ denote the 
set of polynomials in $q^2$ with coefficients in $\RR^+.$   Let $m_{\lambda}\in \Lambda_{\QQ}^n$ 
be the monomial symmetric function indexed by $\lambda \vdash n$. 
In this paper, we prove the following result which shows monotonicity of the coefficient 
of $(-1)^rx^{n-r}$ for $r=0,1,\ldots,n$ in $\zeta_{m_{\lambda}}^{\sL_{T}^q}(x)$ 
when we go up along $\GTS_n$. 

\begin{theorem}
\label{thm:main}
Let $T_1$ and $T_2$ be two trees with $n$ vertices such that $T_2$ covers $T_1$ 
in $\GTS_n$.  Let $\sL_{T_1}^q$ and $\sL_{T_2}^q$ be the $q$-Laplacians 
of $T_1$ and $T_2$ respectively.  For $\lambda \vdash n$, let 
\begin{eqnarray*}
  \zeta^{\sL_{T_1}^q}_{m_{\lambda}}(x) & = & d_{m_{\lambda}}(xI -\sL_{T_1}^q) = \sum_{r=0}^n (-1)^r 
c_{m_{\lambda},r}^{\sL_{T_1}^q}(q) x^{n-r} \mbox{ and} \\
\zeta^{\sL_{T_2}^q}_{m_{\lambda}}(x) & = & d_{m_{\lambda}}(xI -\sL_{T_2}^q) = \sum_{r=0}^n (-1)^r 
c_{m_{\lambda},r}^{\sL_{T_2}^q}(q) x^{n-r}.
\end{eqnarray*}
Then for all $\lambda \vdash n$, we assert that  
$c_{m_{\lambda},r}^{\sL_{T_1}^q}(q) - c_{m_{\lambda},r}^{\sL_{T_2}^q}(q) 
\in \RR^+[q^2]$, where $r=0,1,\ldots,n$.   Further if $\lambda \neq 2^k,1^{n-2k}\vdash n$, 
then, $d_{m_{\lambda}}(xI -\sL_{T_1}^q)=d_{m_{\lambda}}(xI -\sL_{T_2}^q)=0.$
\end{theorem}

Recall that for $\gamma \in \Lambda_{\QQ}^n$, $\Gamma_{\gamma} = \ch^{-1}(\gamma)$. 
For the proof of Theorem \ref{thm:main}, we show the following lemma involving 
$\Gamma_{m_{\lambda}}=\ch^{-1}(m_{\lambda})$ and binomial coefficients.
Let $\Gamma_{\gamma}(j)$ denote the class function  $\Gamma_{\gamma}(\cdot)$ evaluated at a 
permutation $\psi \in \SSS_n$ with cycle type $2^j,1^{n-2j}$.  For $0 \leq i \leq \nhalf$, define
\begin{equation}
  \alpha_i(\gamma) = \sum_{j=0}^i \binom{i}{j} \Gamma_{\gamma}(j). \label{eqn:main}
\end{equation}

We prove the following lemma which we believe is of independent interest.

\begin{lemma}
\label{lem:frob_inv_monomial}
For all $\lambda \vdash n$ and for  $0\leq i \leq \nhalf$, 
the quantity  
$\alpha_{i}^{}(m_{\lambda})$ is a non-negative integral multiple of  $2^i$. 
Moreover  for $0\leq i \leq \nhalf$,  
$\alpha_{i}^{}(m_{\lambda})  = 2^i$ if $\lambda=2^i,1^{n-2i}$ and 
$0$ otherwise.
\end{lemma}

\section{Preliminaries}

We now give some motivational background for our results and place it in its context.  
The normalized immanant
of a matrix $A$ corresponding to a partition $\lambda$ is 
defined as $\displaystyle \od_{\lambda}(A) = \frac{d_{\lambda}(A)}{ \chi_{\lambda}^{}(\id)}$.
Here, $\chi_{\lambda}^{}(\id)$ equals the dimension of the irreducible representation
of $\SSS_n$ over $\CC$ indexed by $\lambda$.

Schur \cite{schur-immanant-ineqs} showed 
that among normalized immanants, the determinant is the smallest 
normalized immanant for any positive semidefinite Hermitian matrix.
This shows that for any positive semidefinite Hermitian matrix $A$, all its
normalized immanants (and hence immanants) are non-negative.
Though the immanant $d_{\lambda}(A)$ is non-negative, \eqref{eqn:immanant}
is not a non-negative expression for $d_{\lambda}(A)$ as 
$\chi_{\lambda}^{}(\psi) \prod_{i=1}^n a_{i,\psi(i)}$ is not necessarily 
non-negative for all
$\psi \in \SSS_n$ that contribute to $d_{\lambda}(A)$.  

 Recall that the
Laplacian matrix $L_G$ of a graph $G$ is defined as $L_G = D-A$, 
where $D$ is the diagonal matrix with degrees of $G$ on the diagonal and $A$ 
is the adjacency matrix of $G$.  It is well known 
that $L_G$ is positive semidefinite for all graphs $G$ (see \cite{godsil-royle}).
When the matrix $A$ is the  Laplacian $L_T$ of a tree $T$,
then Chan and Lam in \cite{hook_immanant_explained-chan_lam} gave the
following two results which give an alternate positive expression for the immanant $d_{\lambda}(L_T)$.

We denote the number of parts of a partition $\lambda$ of $n$  as $l(\lambda)$.
For $\lambda \vdash n$, let $\chi_{\lambda}^{}(j)$ denote the character value 
$\chi_{\lambda}^{}(\cdot)$
evaluated at a permutation $\psi \in \SSS_n$ with cycle type $2^j, 1^{n-2j}$.
For $0 \leq i \leq \nhalf$ and $\lambda \vdash n$, define 
\begin{equation}
  \alpha_{i,\lambda} = \sum_{j=0}^i \binom{i}{j} \chi_{\lambda}^{}(j). \label{eqn:main_char}
\end{equation}

\begin{lemma}[Chan and Lam, \cite{chan-lam-binom-coeffs-char}]
\label{lem:chan_lam_bino_char}
For all $\lambda \vdash n$ and for  $0\leq i \leq \nhalf$, the quantity  
$\alpha_{i, \lambda}$ is a non-negative integral multiple of $2^i$.  
Moreover  for $1\leq i \leq \nhalf$, 
$\alpha_{i,\lambda}=0$ if and only if  $l(\lambda)>n-i$. 
\end{lemma}

\begin{theorem}[Chan and Lam, \cite{hook_immanant_explained-chan_lam}]
  \label{thm:positive_immanant}
 Let $L_T$ be the Laplacian matrix of a tree $T$ on $n$ vertices. Then,
 for all $\lambda \vdash n$, 
 $d_{\lambda}(L_T) = \sum_{i=0}^{\nhalf} \alpha_{i, \lambda} a_i(T)$,
 where $a_i(T)$ equals the number of vertex orientations with exactly 
 $i$  bidirected edges (and is hence a non-negative integer for all $i$).
\end{theorem}

Lemma \ref{lem:chan_lam_bino_char} and Theorem \ref{thm:positive_immanant} make it clear that all immanants 
of  $L_T$ are non-negative.  Similar results are known for the $q$-Laplacian of $T$.  Define 
$\sL_G^q = I+q^2(D-I) -qA$ as the $q$-Laplacian of a graph $G$,  where $D$ and $A$ are as before and $I$ is 
the identity matrix.  Here $q$ is a real variable.   It is easy to see that
setting $q=1$ in $\sL_G^q$ gives us the usual combinatorial Laplacian $L_G$.
The matrix $\sL_G^q$ has appeared in the contexts
of Ihara- Selberg zeta functions of graphs $G$ (see Bass \cite{bass} and  Foata and Zeilberger \cite{foata-zeilberger-bass-trams}). When graph $G$ is a tree $T$,  $\sL_T^q$ has connections with  
inverse of the exponential distance matrix of $T$ (see Bapat, Lal and Pati 
\cite{bapat-lal-pati} and Nagar \cite{nagar}).  
Nagar and  Sivasubramanian in \cite{mukesh-siva-hook} gave the following $q$-analogue of Theorem
\ref{thm:positive_immanant} thereby generalising Theorem 
\ref{thm:positive_immanant}  to $\sL_T^q$, the $q$-Laplacian 
of $T$. 

\begin{theorem}[Nagar and Sivasubramanian]
  \label{thm:mukesh-siva-positive_immanant}
 Let $\sL_T^q$ be the $q$-Laplacian matrix of a tree $T$ on $n$ vertices. Then,
 for all $\lambda \vdash n$, 
 $d_{\lambda}(\sL_T^q) = \sum_{i=0}^{\nhalf} \alpha_{i, \lambda} a_i(T,q)$,
 where the $a_i(T,q)$ is a polynomial in $q^2$, counting vertex 
 orientations with exactly  $i$  bidirected edges with respect to a
 statistic $\law(\cdot)$.
\end{theorem}

As mentioned earlier, setting $q=1$ in $\sL_T^q$ gives $L_T$ and 
clearly from Theorem \ref{thm:mukesh-siva-positive_immanant},
setting $q=1$ in $a_i(T,q)$ gives $a_i(T)$ (see \cite{mukesh-siva-hook} for the  definitions) 
for all $i$ with $0\leq i \leq \nhalf$.    
Further, since $a_i(T,q)$ is a polynomial
in $q^2$, Theorem   \ref{thm:mukesh-siva-positive_immanant} holds for
all $q \in \RR$.  Extensions of this result to the $q,t$-Laplacian denoted as $\sL_T^{q,t}$ are 
also given in \cite{mukesh-siva-hook} and a counterpart of Theorem
\ref{thm:mukesh-siva-positive_immanant} holds even when
$q,t \in \RR$ with $qt > 0$ and when $q,t \in \CC$ with $qt > 0$.
Later, in \cite{mukesh-siva-gts}, the following more general result about the 
coefficient of $(-1)^rx^{n-r}$ in $d_{\lambda}(xI-\sL_T^q)$  was proved.

\begin{lemma}[Nagar and Sivasubramanian]
	\label{lem:coeff_imm_poly}
	Let $T$ be a tree on $n$ vertices with $q$-Laplacian $\sL_T^q$. 
	For $\lambda \vdash n$, let 
	$d_{\lambda}(xI-\sL_T^q)=\sum_{r=0}^n (-1)^r 
	c_{\lambda,r}^{\sL_{T}^q}(q) x^{n-r}$. Then for $0\leq r \leq n$, we have 
	$c_{\lambda,r}^{\sL_{T}^q}(q)
	=\sum_{i=0}^{\rhalf}\alpha_{i,\lambda}^{}a_{i,r}^{}(T,q)$, 
	where $a_{i,r}^{}(T,q)$ is a polynomial in $q^2$, counting vertex 
orientations on all vertex subsets $B$ having exactly $r$ vertices in $T$ and with exactly  
$i$  bidirected edges with respect to a
	statistic $\awy_B^T(\cdot)$.
\end{lemma}

In Lemma \ref{lem:coeff_imm_poly}, setting $r=n$ in $ a_{i,r}^{}(T,q)$ clearly 
gives $ a_{i}^{}(T,q)$  for all $i$ (see \cite{mukesh-siva-gts} for the  definition of $ a_{i,r}^{}(T,q)$). 
Thus, Lemma \ref{lem:coeff_imm_poly} is a generalization of Theorem \ref{thm:mukesh-siva-positive_immanant}. 
Using Lemma \ref{lem:coeff_imm_poly}, the following
result  about monotonicity for the coefficients of all
immanantal polynomials of $\sL_{T}^q$ along $\GTS_n$ was proved in \cite[Theorem 1]{mukesh-siva-gts}.

\begin{theorem}[Nagar and Sivasubramanian]
\label{thm:main_earlier}
Let $T_1$ and $T_2$ be two trees with $n$ vertices such that $T_2$ covers $T_1$ 
in $\GTS_n$.  Let $\sL_{T_1}^q$ and $\sL_{T_2}^q$ be the $q$-Laplacians 
of $T_1$ and $T_2$ respectively.  For $\lambda \vdash n$, let 
\begin{eqnarray*}
  \zeta^{\sL_{T_1}^q}_{s_{\lambda}}(x) & = & d_{\lambda}(xI -\sL_{T_1}^q) = \sum_{r=0}^n (-1)^r 
c_{\lambda,r}^{\sL_{T_1}^q}(q) x^{n-r} \mbox{ and} \\
\zeta^{\sL_{T_2}^q}_{s_{\lambda}}(x) & = & d_{\lambda}(xI -\sL_{T_2}^q) = \sum_{r=0}^n (-1)^r 
c_{\lambda,r}^{\sL_{T_2}^q}(q) x^{n-r}.
\end{eqnarray*}
Then for all $\lambda \vdash n$, we have 
$c_{\lambda,r}^{\sL_{T_1}^q}(q) - c_{\lambda,r}^{\sL_{T_2}^q}(q) 
\in \RR^+[q^2]$, where $r=0,1,\ldots,n$.   
\end{theorem}

The proof of Theorem \ref{thm:main_earlier} involved giving a combinatorial expression for 
the coefficients $c_{\lambda,r}^{\sL_{T_1}^q}(q)$ and $c_{\lambda,r}^{\sL_{T_2}^q}(q)$ (see Lemma \ref{lem:coeff_imm_poly}) and then giving 
an injection between the objects counting $c_{\lambda,r}^{\sL_{T_2}^q}(q)$
and those counting $c_{\lambda,r}^{\sL_{T_1}^q}(q)$ (see Lemma  \ref{lem:diff_a_k,r(T,q)}).

\section{GMF of $q$-Laplacians arising from symmetric functions}
\label{sec:prelims}

Let $\Gamma_{\gamma}=\ch^{-1}(\gamma)$ be the inverse Frobenius
image of the symmetric function $\gamma \in \Lambda_{\QQ}^n$. 
Note that $\Gamma_{\gamma}$ is a class function of $\SSS_n$ over $\CC$ indexed by $\gamma$. 
Recall \eqref{eqn:gen_matrix_fn}, the generalized matrix function of the matrix $\sL_T^q=(l_{i,j}^q)_{1\leq i,j \leq n }$ with respect
to the symmetric function $\gamma$ is defined as 

\begin{equation}
  \label{eqn:gen_matrix_fn_sL}
  d_{\gamma}(\sL_T^q) = \sum_{\psi \in \SSS_n} \Gamma_{\gamma}(\psi) \prod_{i=1}^n l_{i,\psi(i)}^q.
\end{equation}

For the matrix $\sL_T^q$, it is very easy to show the following
counterpart of Theorem \ref{thm:mukesh-siva-positive_immanant}.  Since
the proof is a verbatim copy, we omit it.

\begin{theorem}
  \label{thm:mukesh-siva-gen_matrix_fn}
 Let $\sL_T^q$ be the $q$-Laplacian matrix of a tree $T$ on $n$ vertices. 
 Then, for all $\gamma \in \Lambda_{\QQ}^n$, 
 $d_{\gamma}(\sL_T^q) = \sum_{i=0}^{\nhalf} \alpha_i (\gamma) a_i(T,q)$.
\end{theorem}

It is very simple to show the following
counterpart of Lemma \ref{lem:coeff_imm_poly}.  Since
the proof is identical, we omit it and merely state the result.

\begin{lemma}
	\label{lem:coeff_gmf_poly}
	Let $T$ be a tree on $n$ vertices with $q$-Laplacian $\sL_T^q$. 
	Then for all $\gamma \in \Lambda_{\QQ}^n$, we have 
	$d_{\gamma}(xI-\sL_T^q)=\sum_{r=0}^n (-1)^r 
	c_{\gamma,r}^{\sL_{T}^q}(q) x^{n-r}$, where 
	$c_{\gamma,r}^{\sL_{T}^q}(q)
	=\sum_{i=0}^{\rhalf}\alpha_{i}(\gamma)a_{i,r}^{}(T,q)$ for all $r=0,1,\ldots,n$.
\end{lemma}

Every element $\gamma \in \Lambda_{\QQ}^n$ is written as a linear 
combinatation of basis vectors of $\Lambda_{\QQ}^n$.
The vector space $\Lambda_{\QQ}^n$ has six standard bases.
We use standard terminology to 
denote each of the usual bases of $\Lambda_{\QQ}^n$.  Thus,
$s_{\lambda}$, $p_{\lambda}$, $e_{\lambda}$, $h_{\lambda}$, 
$m_{\lambda}$ and $f_{\lambda}$  denote the
Schur,  power sum, elementary, homogenous, monomial 
and forgotten symmetric functions respectively. 
The inverse Frobenius map of each of these will be denoted by the
same letter, but in capital font and with the partition $\lambda$ 
as a superscript
rather than a subscript.  Thus, $E^{\lambda} = \ch^{-1}(e_{\lambda})$,
$\chi_{\lambda}^{} = S^{\lambda} = \ch^{-1}(s_{\lambda})$ and so on. 

As mentioned earlier, the inverse Frobenius
image $\ch^{-1}(s_{\lambda})$ of the Schur symmetric function $s_{\lambda}$ 
is $\chi_{\lambda}^{}$,
the irreducible character of $\SSS_n$ over $\CC$ indexed by $\lambda$.
Thus, if any symmetric function $\gamma \in \Lambda_{\QQ}^n$ is schur-positive
(that is, $\gamma = \sum_{\lambda \vdash n} a_{\lambda} s_{\lambda}$ where
$a_{\lambda} \in \RR^{+}$ for all $\lambda \vdash n$), 
then, by linearity, Lemma \ref{lem:chan_lam_bino_char} 
will be true with $\chi_{\lambda}^{}$ replaced by $\ch^{-1}(\gamma)$ in \eqref{eqn:main_char}.   
Unfortunately, the monomial symmetric function $m_{\lambda}$ 
indexed by $\lambda$ is not schur-positive.
Thus, if $M^{\lambda} = \ch^{-1}(m_{\lambda})$, it is not clear
that Lemma \ref{lem:chan_lam_bino_char} 
with $\chi_{\lambda}^{}$ replaced by $M^{\lambda}$ 
in \eqref{eqn:main_char} is true.  
Therefore the identity given in \eqref{eqn:main_char} is an special case 
of the identity given in \eqref{eqn:main} when $\gamma=s_{\lambda}$.

For $\gamma \in \Lambda_{\QQ}^n$, let $\Gamma_{\gamma} = \ch^{-1}(\gamma)$. 
In $\alpha_i^{}(\gamma)$, since we only need to evaluate  $\Gamma_{\gamma}(\psi)=\Gamma_{\gamma}(j)$ for 
a permutation $\psi \in \SSS_n$ with cycle type $2^j, 1^{n-2j}$.  
Thus, if $\Gamma_{\gamma}(j) \geq 0$
for all $j$, then from \eqref{eqn:main}, we get 
$\alpha_i(\gamma) \geq 0$.   Since the inverse Frobenius image $\Gamma_{p_{\lambda}}$ of
$p_{\lambda}$ is a scalar multiple of the indicator function of 
the conjugacy class $C_{\lambda}$ indexed by $\lambda$ (see \cite{EC2}), it follows
that $\alpha_i^{}(p_{\lambda}) \geq 0$ for all $\lambda \vdash n$ and for all $i=0,1,\ldots,\nhalf$.

As mentioned earlier, if $\gamma \in \Lambda_{\QQ}^n$
is schur-positive, then $\alpha_i^{}(\gamma) \geq 0$.  Since $h_{\lambda}$
is schur-positive (see \cite[Corollary 7.12.4]{EC2}), it follows from
Lemma \ref{lem:chan_lam_bino_char} that 
$\alpha_i^{}(h_{\lambda}) \geq 0$ for all $\lambda \vdash n$.  
Similarly, it is well known
that $e_{\lambda}$ is also schur-positive.  Thus 
$\alpha_i^{}(e_{\lambda}) \geq 0$ for all $\lambda \vdash n$.  
We record these facts below for future use.

\begin{lemma}
  \label{lem:p_h_e_sum}
For all $\lambda \vdash n$ and for $0 \leq i \leq \nhalf$, we assert that 
$\alpha_i^{}(f) \geq 0$ for 
$f \in \{p_{\lambda}, h_{\lambda}, e_{\lambda}  \}.$
\end{lemma}

With Lemma \ref{lem:p_h_e_sum} in place, it is not hard to show the following.

\begin{corollary}
  \label{cor:three}
Theorem \ref{thm:main_earlier} is true  when we replace
immanantal polynomial $d_{\lambda}(xI - \sL_{T_j}^q)$ by 
the generalized matrix polynomials   
$d_{f}(xI - \sL_{T_j}^q)$ for 
$f \in \{p_{\lambda}, h_{\lambda}, e_{\lambda}  \}$ and 
for $j=1,2$.
\end{corollary}

By Theorem \ref{thm:main_earlier} and Corollary \ref{cor:three}, 
we thus have monotonicity results about four of the six standard bases,  
as we move up on $\GTS_n$.
This paper plugs the gaps left by considering the
generalized matrix polynomials of $\sL_T^q$ arising from the last two 
bases $m_{\lambda}$ and $f_{\lambda}$. In other words, 
we consider the cases when we replace
immanantal polynomial $d_{\lambda}(xI - \sL_{T_j}^q)$ by 
$d_{g}(xI - \sL_{T_j}^q)$ when 
$g \in \{m_{\lambda}, f_{\lambda} \}$  
in Theorem \ref{thm:main_earlier}.

\section{GMF of tree Laplacians arising from $m_{\lambda}$, the monomial symmetric function}

We prove Theorem \ref{thm:main} in this Section.  
We need the notion of 
$\lambda$-brick tabloids of shape $\mu$ which is used to give a combinatorial interpretation of 
the quantity $\Gamma_{m_{\lambda}}(\psi)=M^{\lambda}(\psi)$, where 
$\lambda, \mu\vdash n$ and $\psi\in\SSS_n.$ 
We recall the following definition of 
$\lambda$-brick tabloid of shape $\mu$ defined by  E\u{g}ecio\u{g}lu  
and Remmel in \cite{remmel-monomial-sym-frob}. 

\begin{definition}
Let 
$\lambda \vdash n$ with $\lambda =\lambda_1\geq \lambda_2\geq\cdots\geq \lambda_{l(\lambda)}^{}$. 
Let $\mu \vdash n$  
and let $F_{\mu}$ be its Ferrers diagram. 
A $\lambda$-brick tabloid  $\sB_{\lambda,\mu}$ of shape $\mu$ is a filling of
$F_{\mu}$ with bricks $b_1,b_2,\ldots,b_{l(\lambda)}^{}$ 
of size $\lambda_1, \lambda_2,\ldots, \lambda_{l(\lambda)}^{}$ respectively 
such that brick $b_i$ covers exactly $\lambda_i$ squares of $F_{\mu}$ 
all of which lie in a single row of $F_{\mu}$ and no two bricks overlap. 
Here, bricks of the same size are indistinguishable. 	
\end{definition}

We refer the reader to the book by Mendes and Remmel  \cite[Chapter 2]{mendes-remmel-book} for
an introduction to $\lambda$-brick tabloids. 
Given a brick $b$ in $\sB_{\lambda,\mu}$, let $|b|$ denote the length of $b$. 
Define  $\wt_{\sB_{\lambda,\mu}}^{}(b)$ to be $|b|$ 
if $b$ is at the end of a row in $\sB_{\lambda,\mu}$ and $1$ otherwise. 
We next define a weight $w(\sB_{\lambda,\mu})$ for each 
$\lambda$-brick tabloid $\sB_{\lambda,\mu}$ of shape $\mu$ as follows: 
\begin{equation*}
\label{eqn:def_wt_bricks}
w(\sB_{\lambda,\mu})=\prod_{b\in \sB_{\lambda,\mu}} \wt_{\sB_{\lambda,\mu}}^{}(b).
\end{equation*}

In other words, $w(\sB_{\lambda,\mu})$ is the product of the 
lengths of the rightmost brick in each row of $\sB_{\lambda,\mu}$.
For $\lambda, \mu \vdash n$,  
let $\BT_{\lambda,\mu}$ be the set of all $\lambda$-brick tabloids of shape $\mu$. 
For $\lambda \vdash n$ and for $\psi \in \SSS_n$, 
 E\u{g}ecio\u{g}lu  and Remmel in 
\cite[Theorem 1]{remmel-monomial-sym-frob} gave the following 
combinatorial interpretation of $M^{\lambda}(\psi)$
which we need. 

\begin{theorem}[E\u{g}ecio\u{g}lu  and Remmel]
	\label{thm:main_thm_remmel_egec}
	For $\lambda \vdash n$, let  $M^{\lambda}=\ch^{-1}(m_{\lambda})$. 
	When a permutation $\psi \in \SSS_n$ has cycle type $\mu$, then 
	$$M^{\lambda}(\psi)=(-1)^{l(\lambda)-l(\mu)}\sum_{\sB_{\lambda,\mu} 
\in \BT_{\lambda,\mu}}w(\sB_{\lambda,\mu}).$$
\end{theorem}

From the definition of $\sB_{\lambda,\mu}$, it is easy to see that 
if $\lambda$ is not a refinement of $\mu$,  then $\BT_{\lambda,\mu}=\emptyset$. 
Thus $M^{\lambda}(\psi)=0$, unless when $\lambda$ refines the cycle type $\mu$ of $\psi$. 
Let $M^{\lambda}(j)$ denote the value of the function $M^{\lambda}(\cdot)$ 
evaluated at a permutation $\psi \in \SSS_n$ with cycle type $2^j,1^{n-2j}$.  
We start with the following simple consequence of Theorem \ref{thm:main_thm_remmel_egec}, 

\begin{corollary}
	\label{cor:remmel_mono}
	Let $\lambda \vdash n$ and let $\psi \in \SSS_n$ such that 
	the cycle type of $\psi$ is $\mu=2^j,1^{n-2j}$. Then, 
	$$M^{\lambda}(j)= \left\{  
	\begin{array}{c c}
	(-1)^{j-k}2^k\binom{j}{k} & \mbox{ if } \lambda=2^k,1^{n-2k} \mbox{ and } k\leq j  \\
	0 &  \mbox{ otherwise }
	\end{array}.
	\right.
	$$
\end{corollary}

\begin{proof}
We divide the proof into two cases when $\lambda$ is  a 
refinement of $\mu$ and when $\lambda$ is not a refinement of $\mu$. 
If $\lambda$ is a refinement of $\mu$, then we must have  
$\lambda=2^k,1^{n-2k}$ for some 
$k$ with $0 \leq k\leq j.$ In this case, all the $k$ bricks of 
length $2$ must be  placed in some of the $k$ rows from the first $j$ 
rows of $F_{\mu}$. Thus, the number of 
such $\lambda$-brick tabloids $\sB_{\lambda,\mu}$ of 
shape $\mu$ is $\binom{j}{k}$, and each $\sB_{\lambda,\mu}$  
contributes the weight $2^k$ in the summation of   Theorem \ref{thm:main_thm_remmel_egec}.
Thus,  the total contribution in  $M^{\lambda}(j)$ is  $2^k\binom{j}{k}$ and   $l(\lambda)-l(\mu)=j-k$. Hence  
$M^{\lambda}(j)= (-1)^{j-k}2^k\binom{j}{k}$.

We next consider the case when $\lambda$ is not a refinement of $\mu$.   
By Theorem \ref{thm:main_thm_remmel_egec},  
$M^{\lambda}(j)=0$,  completing the proof. 
\end{proof}

\vspace{2mm} 
 
We next prove Lemma \ref{lem:frob_inv_monomial} which says that for all $\lambda \vdash n$ 
and for $0\leq i \leq \nhalf$, the quantity 
$\alpha_{i}^{}(m_{\lambda})$ is a non-negative integral  multiple of $2^i$.

\vspace{2mm}

\begin{proof}(of Lemma 2)
By Corollary \ref{cor:remmel_mono},  when $\lambda\neq 2^k,1^{n-2k}$, 
$M^{\lambda}(j)= 0$ for all $j$. Thus, from \eqref{eqn:main}, 
$\alpha_{i}^{}(m_{\lambda})=0$ for all $i=0,1,\ldots,\nhalf.$ 
When $\lambda= 2^k,1^{n-2k}\vdash n$ for some $k$ 
with $0\leq k \leq \nhalf$, by Corollary \ref{cor:remmel_mono}, we see that 
	\begin{eqnarray*}
		\alpha_{i}^{}(m_{\lambda}) & = & \sum_{j=0}^{i}M^{\lambda}(j)\binom{i}{j} 
		=  
		\sum\limits_{j=0}^{i}(-1)^{j-k}2^k\binom{j}{k}\binom{i}{j}   \\
		& = & 
		2^k \binom{i}{k}\sum\limits_{j=k}^{i-k}(-1)^{j-k}\binom{i-k}{j-k} 
		=  \left\{  
		\begin{array}{c c}
			2^i & \mbox{ if } i=k  \\
			0 &  \mbox{ otherwise }
		\end{array} .
		\right.
	\end{eqnarray*} 
	
	The proof is complete.
\end{proof}

\vspace{2mm}

Let the tree $T$ have $n$ vertices and for $0\leq i,r\leq n$, 
let $a_{i,r}^{}(T,q)=a_{i,r}^T(q)$ be the polynomial in $q^2$ with non-negative real coefficients, 
defined by Nagar and Sivasubramanian in \cite[page 7]{mukesh-siva-gts}.
They showed the following result  (see 
\cite[Lemmas 19, 23 and Corollary 22]{mukesh-siva-gts}). 

\begin{lemma}[Nagar and Sivasubramanian]
	\label{lem:diff_a_k,r(T,q)}
	Let $T_1$ and $T_2$ be two trees with $n$ vertices such that  $T_2$ covers $T_1$ in $\GTS_n$. 
	Then for all $0\leq i,r\leq n$, we assert that  $a_{i,r}^{}(T_1,q)-a_{i,r}^{}(T_2,q) \in \RR^+[q^2]$. 
\end{lemma}

Using Lemmas \ref{lem:frob_inv_monomial} and  \ref{lem:diff_a_k,r(T,q)}, we can now prove Theorem \ref{thm:main}.

\vspace{2mm}
\begin{proof}(of Theorem 1) 
From Lemmas \ref{lem:frob_inv_monomial} and  \ref{lem:coeff_gmf_poly}, for $0 \leq r \leq n$ and for $j=1,2$, 
it is simple to  see that the coefficient of $(-1)^rx^{n-r}$ in $d_{m_{\lambda}}(xI-\sL_{T_j}^q)$ is given by 
\begin{equation}
\label{eqn:coeff_m}
c_{m_{\lambda},r}^{\sL_{T_j}^q}(q) 
=\sum_{i=0}^{\rhalf}\alpha_{i}(m_{\lambda})a_{i,r}^{}(T_j,q)
= \left\{  
\begin{array}{c c}
2^k a_{k,r}^{}(T_j,q)& \mbox{ if }  \lambda=2^k,1^{n-2k} \mbox{ and } k\leq \rhalf \\
0 &  \mbox{ otherwise }
\end{array} .
\right.
\end{equation}

By Lemma \ref{lem:diff_a_k,r(T,q)} and \eqref{eqn:coeff_m}, we get
$c_{m_{\lambda},r}^{\sL_{T_1}^q}(q)-c_{m_{\lambda},r}^{\sL_{T_2}^q}(q)\in \RR^+[q^2]$ for all $\lambda\vdash n$ 
and for all $r=0,1,\ldots, n.$ 
Furthermore, if  $\lambda\neq 2^k,1^{n-2k} \vdash n$ for all  $k$, then  from \eqref{eqn:coeff_m}, 
the  coefficient of $(-1)^rx^{n-r}$ in $d_{m_{\lambda}}(xI-\sL_{T_j}^q)$ is zero for all $r=0,1,\ldots, n$ and for $j=1,2$. 
The proof is complete.
\end{proof}

\vspace{2mm}

Thus, for all $\lambda\vdash n$ and for all $q\in \RR$,   going up on $\GTS_n$ decreases the 
coefficient of $(-1)^rx^{n-r}$ in $d_{m_{\lambda}}(xI-\sL_{T}^q)$, where $r=0,1,\ldots, n$. 
Consequently, for all positive integers $n$,  this monotonicity result on $\GTS_n$ shows that 
max-min pair of these coefficients is  $(P_n,S_n)$, where $P_n$ 
and $S_n$ are the path tree and the star tree on $n$ vertices respectively.
The following example illustrates Lemma \ref{lem:frob_inv_monomial}

\begin{example} 
	Let $n=15$. For all $\lambda\vdash n$ and for all $i\in \{0,1,2,3,4,5,6,7\}$, 
the values of  the quantity $\alpha_i(m_{\lambda})$ are tabulated in Table \ref{tab:alpha_r_m_lambda}. 
	
\begin{table}[h]
	$$ 	\begin{array}{|l||c|c|c|c|c|c|c|c|}
	\hline
	\lambda \vdash n=15 & i=0 & i=1 & i=2 & i=3  & i=4 & i=5 & i=6 & i=7	 \\
	\hline 
	\hline 
	\lambda= 1^{15} & 1 &	0 & 0 & 0 & 0 & 0 & 0 & 0\\
	\hline
	\lambda= 2,1^{13} & 0 &	2 & 0 & 0 & 0 & 0 & 0 & 0\\
	\hline
	\lambda= 2^2,1^{11} & 0 &	0 & 4 & 0 & 0 & 0 & 0 & 0\\
	\hline
	\lambda= 2^3,1^{9} & 0 &	0 & 0 & 8 & 0 & 0 & 0 & 0\\
	\hline
	\lambda= 2^4,1^{7} & 0 &	0 & 0 & 0 & 16 & 0 & 0 & 0\\
	\hline
	\lambda= 2^5,1^{5} & 0 &	0 & 0 & 0 & 0 & 32 & 0 & 0\\
	\hline
	\lambda= 2^6,1^{3} & 0 &	0 & 0 & 0 & 0 & 0 & 64 & 0\\
	\hline
	\lambda= 2^7,1 & 0 &	0 & 0 & 0 & 0 & 0 & 0 & 128\\
	\hline
	\lambda\neq 2^k,1^{n-2k} & 0 &	0 & 0 & 0 & 0 & 0 & 0 & 0\\
	\hline
	\end{array} $$
	\caption{The value of $\alpha_i(m_{\lambda})$.}
		\label{tab:alpha_r_m_lambda}
\end{table}

\end{example}

\section{GMF of tree Laplacians arising from $f_{\lambda}$, the forgotten symmetric function}

Let $f_{\lambda}$ denote the forgotten 
symmetric function indexed by  $\lambda \vdash n$ and let 
$F^{\lambda}=\ch^{-1}(f_{\lambda})$ denote its inverse Frobenius 
image. 
This section is devoted to show monotonicity of the coefficient of $(-1)^rx^{n-r}$ in $\zeta_{f_{\lambda}}^{\sL_T^q}(x)=d_{f_{\lambda}}(xI-\sL_T^q)$ when we go up along $\GTS_n$.  
For this, we need the following combinatorial interpretation of $F^{\lambda}(\psi)$ by
E\u{g}ecio\u{g}lu  and Remmel in  \cite[Theorem 8]{remmel-monomial-sym-frob}.  
 
 \begin{theorem}[E\u{g}ecio\u{g}lu  and Remmel]
 	\label{thm:remmel_forgottan_inte}
 Let $\lambda \vdash n$ and let $\psi \in \SSS_n$ be a permutation with cycle type $\mu$. Then, 
 $$F^{\lambda}(\psi)=(-1)^{n-l(\mu)}\sum_{\sB_{\lambda,\mu} \in \BT_{\lambda,\mu}}w(\sB_{\lambda,\mu}).$$
 	\end{theorem}
 
 For $\lambda \vdash n$, let  $F^{\lambda}(j)$ denote the value of the function  
$F^{\lambda}(\cdot)$ evaluated at a permutation $\psi \in \SSS_n$ with cycle type $2^j,1^{n-2j}$. 
By Theorem \ref{thm:remmel_forgottan_inte}, the proof of the following 
corollary is 
identical to the proof of Corollary \ref{cor:remmel_mono}, we omit it and merely state the result.
\begin{corollary}
	\label{cor:remmel_forgottan}
	For $\lambda \vdash n$ and for  $0\leq j \leq \nhalf$,  we have 
	$$F^{\lambda}(j)= \left\{  
	\begin{array}{c c}
	(-1)^{j}2^k\binom{j}{k} & \mbox{ if } \lambda=2^k,1^{n-2k} \mbox{ and } k\leq j  \\
	0 &  \mbox{ otherwise }
	\end{array} .
	\right.
	$$
\end{corollary} 

For a fixed $\lambda \vdash n$, let $\gamma=f_{\lambda}$ in \eqref{eqn:main}. 
We next calculate the quantity $\alpha_{i}(f_{\lambda})$ in the following lemma 
which will be used later in determining the coefficient of $(-1)^rx^{n-r}$ in $d_{f_{\lambda}}(xI-\sL_{T}^q).$ 
\begin{lemma}
	\label{lem:alpha_f}
	Let $\lambda \vdash n$. 
	Then for $0\leq i \leq \nhalf$,  we have 
	$$\alpha_{i}^{}(f_{\lambda})= \left\{  
	\begin{array}{c c}
	(-1)^i 2^i & \mbox{ if } \lambda=2^i,1^{n-2i}  \\
	0 &  \mbox{ otherwise }
	\end{array} .
	\right.
	$$
\end{lemma}

\begin{proof}
By Corollary \ref{cor:remmel_forgottan} when $\lambda\neq 2^k,1^{n-2k}$,  
$F^{\lambda}(j)=0$ for all $j$. By \eqref{eqn:main},  $\alpha_{i}^{}(f_{\lambda})=0$ for 
all $i=0,1,\ldots,\nhalf.$ 
We next assume $\lambda= 2^k,1^{n-2k}\vdash n$ for some $k$ 
with $0\leq k \leq \nhalf$. In this case, by Corollary \ref{cor:remmel_forgottan}, we see that 
\begin{eqnarray*}
\alpha_{i}^{}(f_{\lambda}) & = & \sum_{j=0}^{i}F^{\lambda}(j)\binom{i}{j} 
=  
\sum\limits_{j=0}^{i}(-1)^{j}2^k\binom{j}{k}\binom{i}{j}   \\
& = & 
2^k \binom{i}{k}\sum\limits_{j=k}^{i-k}(-1)^{j}\binom{i-k}{j-k} 
			 = 
	(-1)^k	2^k \binom{i}{k}\sum\limits_{i=0}^{i-k}(-1)^{i}\binom{i-k}{i} \\
	&	= & \left\{  
		\begin{array}{c c}
			(-1)^i	2^i & \mbox{ if } i=k  \\
			0 &  \mbox{ otherwise }
		\end{array} .
		\right.
	\end{eqnarray*} 
	
	All equalities above follow by simple 
	manipulations and from well known binomial identities. The proof is complete.
\end{proof}

\begin{theorem}
	\label{thm:main_f_lambda}
Let $T$ be a tree on $n$ vertices with $q$-Laplacian matrix $\sL_{T}^q$. 
Then  for $0\leq r \leq n$, the coefficient of $(-1)^rx^{n-r}$ in $d_{f_{\lambda}}(xI-\sL_{T}^q)$ is given by 
\begin{equation*}
\label{eqn:coeff_f}
c_{f_{\lambda},r}^{\sL_{T}^q}(q)
=\left\{  
\begin{array}{c c}
(-1)^k2^k a_{k,r}(T,q) & \mbox{ if } \lambda=2^k,1^{n-2k} \mbox{ with } k \leq \rhalf \\
0 &  \mbox{ otherwise }
\end{array}.
\right. 
\end{equation*} 
Furthermore, for all $\lambda\vdash n$
and all $q\in \RR$,
going up on $\GTS_n$ 
decreases $c_{f_{\lambda},r}^{\sL_{T}^q}(q)$ in absolute value.
\end{theorem} 

\begin{proof}
	From Lemmas \ref{lem:coeff_gmf_poly} and  \ref{lem:alpha_f},  
it is simple to see that 	
	\begin{equation*}
	c_{f_{\lambda},r}^{\sL_{T}^q}(q) 
	=\sum_{i=0}^{\rhalf}\alpha_{i}(f_{\lambda})a_{i,r}^{}(T,q)
	= \left\{  
	\begin{array}{c c}
	(-1)^k2^k a_{k,r}^{}(T,q)& \mbox{ if }  \lambda=2^k,1^{n-2k} \mbox{ and } k\leq \rhalf \\
	0 &  \mbox{ otherwise }
	\end{array} .
	\right.
	\end{equation*}
	Let $T_1$ and $T_2$ be two trees with $n$ vertices such that  $T_2$ covers $T_1$ in $\GTS_n$.  
By Lemma \ref{lem:diff_a_k,r(T,q)},  
$\left|c_{f_{\lambda},r}^{\sL_{T_1}^q}(q)\right|-\left|c_{f_{\lambda},r}^{\sL_{T_2}^q}(q)\right|\in \RR^+[q^2]$ 
completing the proof.
\end{proof}

\vspace{2mm}
Thus, for all $\lambda\vdash n$ and for all $ r=0,1,\ldots, n$, 
by Theorems \ref{thm:main} and \ref{thm:main_f_lambda} and Corollary \ref{cor:three}, we get 
that the coefficient of $x^r$ in $d_{\gamma}(xI-\sL_{T}^q)$ 
decreases as we go up along $\GTS_n$ in absolute value for each  
$\gamma\in \{m_{\lambda},s_{\lambda},p_{\lambda},h_{\lambda},e_{\lambda},f_{\lambda}\}$.  Plugging in
 $q=1$, we get the following corollary of 
Theorems \ref{thm:main} and \ref{thm:main_f_lambda} and Corollary \ref{cor:three}. 

\begin{corollary}
For all $\lambda\vdash n$ and for all 
$\gamma\in \{m_{\lambda},s_{\lambda},p_{\lambda},h_{\lambda},e_{\lambda},f_{\lambda}\}$, 
going up on $\GTS_n$ decreases  the  coefficient of $x^r$ in $d_{\gamma}(xI-L_T)$ in absolute value, 
for $r=0,1,\ldots,n$.  Here $L_T$ is the usual combinatorial Laplacian of the tree $T$.
\end{corollary}

\section*{Acknowledgement}
The first author would like to 
acknowledge SERB, Government of India for providing a national postdoctoral
fellowship with the file no. PDF/2018/000828.  
The second author acknowledges support from project grant
15IRCCFS003 given by IIT Bombay and from project MTR/2018/000453 
given by SERB, Government of India.

Our main theorem in this work was in its conjecture form, tested
using the computer package ``SageMath''. We thank the authors for generously releasing SageMath as an open-source package.
\bibliographystyle{acm}
\bibliography{main}

\begin{thebibliography}{10}

\bibitem{bapat-lal-pati}
{\sc Bapat, R.~B., Lal, A.~K., and Pati, S.}
\newblock A $q$-analogue of the distance matrix of a tree.
\newblock {\em Linear Algebra and its Applications 416\/} (2006), 799--814.

\bibitem{bass}
{\sc Bass, H.}
\newblock The {I}hara-{S}elberg zeta function of a tree lattice.
\newblock {\em International Journal of Math. 3\/} (1992), 717--797.

\bibitem{chan-lam-binom-coeffs-char}
{\sc Chan, O., and Lam, T.~K.}
\newblock Binomial {C}oefficients and {C}haracters of the {S}ymmetric {G}roup.
\newblock {\em Technical Report 693\/} (1996), National Univ of Singapore.

\bibitem{hook_immanant_explained-chan_lam}
{\sc Chan, O., and Lam, T.~K.}
\newblock Hook {I}mmanantal {I}nequalities for {T}rees {E}xplained.
\newblock {\em Linear Algebra and its Applications 273\/} (1998), 119--131.

\bibitem{csikvari-poset1}
{\sc Csikv{\'a}ri, P.}
\newblock On a poset of trees.
\newblock {\em Combinatorica 30 (2)\/} (2010), 125--137.

\bibitem{remmel-monomial-sym-frob}
{\sc E\u{g}ecio\u{g}lu, {\"O}., and Remmel, J.~B.}
\newblock The monomial symmetric functions and the {F}robenius map.
\newblock {\em Journal of Combinatorial Theory, Series A 54}, 2 (1990),
  272--295.

\bibitem{foata-zeilberger-bass-trams}
{\sc Foata, D., and Zeilberger, D.}
\newblock Combinatorial {P}roofs of {B}ass's {E}valuations of the
  {I}hara-{S}elberg {Z}eta function of a {G}raph.
\newblock {\em Transactions of the AMS 351\/} (1999), 2257--2274.

\bibitem{godsil-royle}
{\sc Godsil, C., and Royle, G.}
\newblock {\em Algebraic Graph Theory}.
\newblock Springer Verlag, 2001.

\bibitem{mendes-remmel-book}
{\sc Mendes, A., and Remmel, J.~B.}
\newblock {\em {C}ounting with {S}ymmetric {F}unctions}.
\newblock Developments in mathematics. Springer, Cham, 2015.

\bibitem{nagar}
{\sc Nagar, M.~K.}
\newblock Eigenvalue monotonicity of $q$-{L}aplacians of trees along a poset.
\newblock {\em Linear Algebra and its Applications 571\/} (2019), 110--131.

\bibitem{mukesh-siva-hook}
{\sc Nagar, M.~K., and Sivasubramanian, S.}
\newblock Hook immanantal and {H}adamard inequalities for $q$-{L}aplacians of
  trees.
\newblock {\em Linear Algebra and its Applications 523\/} (2017), 131--151.

\bibitem{mukesh-siva-gts}
{\sc Nagar, M.~K., and Sivasubramanian, S.}
\newblock Laplacian immanantal polynomials and the {G}{T}{S} poset on trees.
\newblock {\em Linear Algebra and its Applications 561\/} (2019), 1--23.

\bibitem{schur-immanant-ineqs}
{\sc Schur, I.}
\newblock {\"U}ber endliche {G}ruppen und {H}ermitesche {F}ormen.
\newblock {\em Math. Z. 1\/} (1918), 184--207.

\bibitem{EC2}
{\sc Stanley, R.~P.}
\newblock {\em Enumerative Combinatorics, vol 2}.
\newblock Cambridge University Press, 2001.

\end{thebibliography}
\end{document}